\documentclass[reqno]{amsart}
\usepackage[bookmarks=false]{hyperref} 
\usepackage{amssymb}
\usepackage{amsmath}
\usepackage{amsthm} 
\usepackage{mathtools}
\usepackage{color} 
\usepackage[svgnames]{xcolor}
\usepackage[utf8]{inputenc}
\usepackage[T1]{fontenc}
\usepackage{graphicx}
\usepackage{placeins}
\usepackage{vmargin}
\usepackage{cite}
\usepackage{setspace}

\hypersetup{
  colorlinks=true,
  linkcolor=Blue  ,          
  citecolor=Red,        
  filecolor=Magenta,      
  urlcolor=Green,           
pdfpagemode=UseOutlines,
pdftitle={On the singularly perturbed derivative Nonlinear Schr\"odinger equation},
pdfauthor={Phan Van Tin <van-tin.phan@univ-tlse3.fr>},
}

\numberwithin{equation}{section}
\numberwithin{figure}{section}

\newtheorem{theorem}{Theorem}[section]
\newtheorem{proposition}[theorem]{Proposition}
\newtheorem{lemma}[theorem]{Lemma}

\newtheorem*{assumption}{Assumption}

\theoremstyle{definition}

\theoremstyle{remark}
\newtheorem{remark}[theorem]{Remark}

\DeclarePairedDelimiter{\norm}{\lVert}{\rVert}

\newcommand{\N}{\mathbb{N}}
\newcommand{\R}{\mathbb{R}}
\newcommand{\C}{\mathbb C}

\renewcommand{\leq}{\leqslant}
\renewcommand{\geq}{\geqslant}

\DeclareMathAlphabet{\mathpzc}{OT1}{pzc}{m}{it}
\renewcommand{\Re}{\mathcal R\!\mathpzc{e}}
\renewcommand{\Im}{\mathcal I\!\mathpzc{m}}

\begin{document}

\title[On the derivative nonlinear Schr\"odinger equation on the half line with Robin boundary condition]{On the derivative nonlinear Schr\"odinger equation on the half line with Robin boundary condition}

\author[Phan Van Tin]{Phan Van Tin}

\address[Phan Van Tin]{Institut de Math\'ematiques de Toulouse ; UMR5219,
  \newline\indent
  Universit\'e de Toulouse ; CNRS,
  \newline\indent
  UPS IMT, F-31062 Toulouse Cedex 9,
  \newline\indent
  France}
\email[Phan Van Tin]{van-tin.phan@univ-tlse3.fr}


\date{\today}

\begin{abstract}
We consider the Schr\"odinger equation with nonlinear derivative term on $[0,+\infty)$ under Robin boundary condition at $0$. Using a virial argument, we obtain the existence of blowing up solutions and using variational techniques, we obtain stability and instability by blow up results for standing waves. 
\end{abstract}

\maketitle
\tableofcontents


\section{Introduction}

In this paper, we consider the derivative nonlinear Schr\"odinger equation on $[0,+\infty)$ with Robin boundary condition at 0: 
\begin{equation}\label{eq first}
\begin{cases}
iv_t+v_{xx}=\frac{i}{2}|v|^2v_x - \frac{i}{2}v^2\overline{v_x}-\frac{3}{16}|v|^4v \quad \text{ for } x \in \R_{+}, \\
v(0) = \varphi, \\
\partial_x v(t,0) = \alpha v(t,0)\quad \forall t\in\R,
\end{cases}
\end{equation}    
where $\alpha \in \R$ is a given constant.

The linear parts of \eqref{eq first} can be rewritten in the following forms:
\begin{equation}\label{linearequation}
\begin{cases}
iv_t+\widetilde{H}_{\alpha}v=0 \text{ for } x \in \R_{+},\\
v(0)=\varphi,
\end{cases}
\end{equation}
where $\widetilde{H}_{\alpha}$ are self-adjoint operators defined by
\begin{align*}
\widetilde{H}_{\alpha} &:  D(\widetilde{H}_{\alpha})\subset L^2(\R_{+})\rightarrow L^2(\R_{+}),\\
\widetilde{H}_{\alpha}u &= u_{xx},\quad  D(\widetilde{H}_{\alpha})=\left\{u \in H^2(\R_{+}) : u_x(0^{+})=\alpha u(0^{+}))\right\}.
\end{align*}
We call $e^{i \widetilde{H}_{\alpha} t}: \R \rightarrow \mathcal{L}(L^2(\R^{+}))$ is group defining the solution of \eqref{linearequation}.   
 
The derivative nonlinear Schr\"odinger equation was originally introduced in Plasma Physics as a simplified model for Alfvén wave propagation. Since then, it has attracted a lot of attention from the mathematical community (see e.g \cite{CoKeStTaTa01,CoKeStTaTa02,Ha93,HaOz92,He06,KaNe78,Ta01,Ta16}). 

Consider the equation \eqref{eq first}, and set
\[
u(t,x)=\exp\left(\frac{3i}{4}\int_{\infty}^x|v(t,y)|^2\, dy\right)v(t,x).
\]
Using the Gauge transformation, we see that $u$ solves
\begin{equation}\label{eqnew}
iu_t+u_{xx}=i\partial_x(|u|^2u), \quad t \in \R, ,\ x \in (0,\infty),
\end{equation}
under a boundary condition $\partial_x u(t,0)=\alpha u(t,0)+\frac{3i}{4}|u(t,0)|^2u(t,0)$. In all line case, there are many papers to deal with Cauchy problem of \eqref{eqnew} (see e.g \cite{HaOz94,TsFu80,TsFu81}). In \cite{HaOz94}, the authors establish the local well posedness in $H^1(\R)$ by using a Gauge transform. Indeed, since $u$ solves \eqref{eqnew} on $\R$, by setting
\begin{align}
h(t,x)&=\exp\left(-i\int_{-\infty}^x|u(t,y)|^2\,dy\right)u(t,x), \nonumber \\
k&= h_x+\frac{i}{2}|h|^2h, \label{12} 
\end{align}      
we have $h,k$ solve
\begin{equation}
\label{eqq}
\begin{cases}
ih_t+h_{xx}= -ih^2\overline{k},\\
ik_t+k_{xx}=ik^2\overline{h}.
\end{cases}
\end{equation}
By classical arguments, we can prove that there exists a unique solution $h,k \in C([0,T],L^2(\R)) \cap L^4([0,T],L^{\infty}(\R))$ given $h_0,k_0 \in L^2(\R)$ are satisfy \eqref{12}. To obtain the existence solution of \eqref{eq first}, the authors prove that the relation \eqref{12} satisfies for all $t \in  [0,T]$. Thus, since $h,k$ solve \eqref{eqq} satisfy \eqref{12}, if we set 
\[
u(t,x)=\exp\left(i\int_{-\infty}^x|h(t,y)|^2\, dy\right)h(t,x),
\]
then $u \in C([0,T],H^1(\R))$ solves \eqref{eq first}. In \cite{bahouri2020global}, the authors have proved the global well posedness of \eqref{eqnew} given initial data in $H^{\frac{1}{2}}(\R)$. In half line case, \cite{Wu13} Wu  prove existence of blow up solution of \eqref{eqnew} under Dirichlet boundary condition, given initial data in $\Sigma := \{u_0 \in H^2(\R^{+}), xu_0 \in L^2(\R^{+})\}$. In this paper, we give a proof of existence of blow up solution of \eqref{eq first} under Robin boundary condition.     
 
To study equation \eqref{eq first}, we start by the definition of solution on $H^1(\R^{+})$. Since \eqref{eq first} contains a Robin boundary condition, the notion of solution in $H^1(\R^{+})$ is not completely clear. We use the following definition. Let $I$ be an open interval of $\R$. We say that $v$ is a $H^1(\R^{+})$ solution of the problem \eqref{eq first} on $I$ if $v \in C(I,H^1(\R^{+}))$ satisfies the following equation
\begin{equation}
\label{midsolution of pb 2}
v(t)= e^{i\widetilde{H}_{\alpha}t}\varphi-i\int_0^t e^{i\widetilde{H}_{\alpha}(t-s)}g(v(s)) \ ds,
\end{equation} 
where $g$ is the function defined by
\[
g(v)= \frac{i}{2}|v|^2v_x-\frac{i}{2}v^2\overline{v}_x-\frac{3}{16}|v|^4v.
\]



Let $v \in C(I,D(\widetilde{H}_{\alpha}))$ be classical solution of \eqref{eq first}. At least formally, we have
\begin{align*}
\frac{1}{2} \partial_t(|v|^2)=-\partial_x\Im(v_x\overline{v}).
\end{align*}
Therefore, using the Robin boundary condition we have 
\begin{align*}
\partial_t\left(\frac{1}{2}\int_0^{\infty}|v|^2 \, dx\right) &= - \Im(v_x\overline{v})(\infty) + \Im(v_x\overline{v})(0) \\
&= \Im(v_x\overline{v})(0) \\
&= \alpha\Im(|v(0)|^2) \\
&=0.
\end{align*}
This implies the conservation of the mass. By elementary calculations, we have
\begin{align*}
\partial_t\left(|v_x|^2-\frac{1}{16}|v|^6\right)&= \partial_x\left(2\Re(v_x\overline{v}_t)-\frac{1}{2}|v|^2|v_x|^2+\frac{1}{2}v^2\overline{v}_x^2\right).
\end{align*}
Hence, integrating the two sides in space, we obtain
\begin{align*}
\partial_t \left(\int_{\R_{+}}|v_x|^2\,dx-\frac{1}{16}|v|^6\,dx\right)&=-2\Re(v_x(0)\overline{v}_t(0))+\frac{1}{2}|v(0)|^2|v_x(0)|^2-\frac{1}{2}v(0)^2\overline{v_x(0)}^2 
\end{align*}
Using the Robin boundary condition for $v$, we obtain
\begin{align*}
\partial_t \left(\int_{\R_{+}}|v_x|^2\,dx-\frac{1}{16}|v|^6\,dx\right)&=-2\alpha\Re(v(0)\overline{v_t(0)})=-\alpha\partial_t(|v(0)|^2).
\end{align*}
This implies the conservation of the energy.

In this paper, we will need the following assumption.

\begin{assumption}\label{assumptionB}
We assume that for all $\varphi \in H^1(\R^{+})$ there exist a solution $v \in C(I,H^1(\R^{+}))$ of \eqref{eq first} for some interval $I \subset \R$. Moreover, $v$ satisfies the following conservation law:
\begin{align*}
M(v)&:=\frac{1}{2}\norm{v}^2_{H^1(\R^{+})}=M(\varphi),\\
E(v)&:=\frac{1}{2}\norm{v_x}^2_{L^2(\R^{+})}-\frac{1}{32}\norm{v}_{L^6(\R^{+})}+\frac{\alpha}{2}|v(0)|^2.
\end{align*}  
\end{assumption}

The existence of blowing up solutions for classical nonlinear Schr\"odinger equations was considered by Glassey \cite{Gl77} in 1977. He introduced a concavity argument based on the second derivative in time of $\norm{xu(t)}^2_{L^2}$ to show the existence of blowing up solutions. In this paper, we are also interested in studying the existence of blowing-up solutions of \eqref{eq first}. In the limit case $\alpha=+\infty$, which is formally equivalent to Dirichlet boundary condition if we write $v(0)=\frac{1}{\alpha}v'(0)=0$. In \cite{Wu13}, Wu proved the blow up in finite time of solutions of \eqref{eq first} with Dirichlet boundary condition and some conditions on the initial data. Using the method of Wu \cite{Wu13} we obtain the existence of blowing up solutions in the case $\alpha>0$, under a weighted space condition for the initial data and negativity of the energy. Our first main result is the following.

\begin{theorem}\label{thm blow up}
We assume assumption \ref{assumptionB}. Let $\alpha>0$ and $\varphi \in \Sigma$ where 
\[
\Sigma = \left\{ u \in D(\widetilde{H}_{\alpha}), xu \in L^2(R_{+}))\right\}
\]
such that $E(\varphi)<0$.
Then the solution $v$ of \eqref{eq first} blows-up in finite time i.e $T_{min}>-\infty$ and $T_{max}<+\infty$. 
\end{theorem}

\begin{remark}\label{remarkonconservationlaw}
In \eqref{eq first}, if we consider nonlinear term $i|v|^2v_x$ instead of $\frac{i}{2}|v|^2v_x-\frac{i}{2}v^2\overline{v_x}-\frac{3}{16}|v|^4v$ then there is no conservation of energy of solution. Indeed, set
\[
u(t,x)=v(t,x) \exp \left(-\frac{i}{4}\int_{\infty}^x|v(t,y)|^2\, dy\right).
\]     
If $v$ solves 
\begin{equation*}
\begin{cases}
iv_t+v_{xx}=i|v|^2v_x,\\
\partial_x v(t,0)=\alpha v(t,0)
\end{cases}
\end{equation*} 
then $u$ solves
\begin{equation}\label{eqofu}
\begin{cases}
iu_t+u_{xx}=\frac{i}{2}|u|^2u_x-\frac{i}{2}u^2\overline{u_x}-\frac{3}{16}|u|^4u,\\
\partial_x u(t,0)=\alpha u(t,0)-\frac{i}{4}|u(t,0)|^2u(t,0).
\end{cases}
\end{equation}
By elementary calculations, since $u$ solves \eqref{eqofu}, we have
\begin{align*}
\partial_t\left(|u_x|^2-\frac{1}{16}|u|^6\right) &= \partial_x\left(2\Re(u_x\overline{u_t})-\frac{1}{2}|u|^2|u_x|^2+\frac{1}{2}u^2\overline{u_x}^2\right).
\end{align*} 
Integrating the two sides in space, we obtain
\[
\partial_t\left(\int_{\R^{+}}|u_x|^2-\frac{1}{16}|u|^6\,dx\right)=-2\Re(u_x(0)\overline{u_t}(0))+\frac{1}{2}|u(0)|^2|u_x(0)|^2-\frac{1}{2}u(0)^2\overline{u_x(0)}^2.
\]
Using the boundary condition of $u$, we obtain
\begin{align*}
\partial_t\left(\int_{\R^{+}}|u_x|^2-\frac{1}{16}|u|^6\,dx\right)&=-2\alpha\Re(u(0)\overline{u_t(0)})-\frac{1}{2}\Im(u(0)|u(0)|^2\overline{u_t(0)})\\
&\quad +\frac{1}{2}|u(0)|^4\left(\alpha^2+\frac{1}{16}|u(0)|^4-\left(\alpha+\frac{i}{4}|u(0)|^2\right)^2\right)\\
&=-\alpha\partial_t(|u(0)|^2)+A,
\end{align*}
where $A=-\frac{1}{2}\Im(u(0)|u(0)|^2\overline{u_t(0)})+\frac{1}{2}|u(0)|^4\left(\alpha^2+\frac{1}{16}|u(0)|^4-\left(\alpha+\frac{i}{4}|u(0)|^2\right)^2\right)$. Moreover, we can not write $A$ in form $\partial_t B(u(0))$, for some function $B: \C \rightarrow \C$. Then, there is no conservation of energy of $u$ and hence, there is no conservation of energy of $v$. 
\end{remark}

The stability of standing waves for classical nonlinear Schr\"odinger equations was originally studied by Cazenave and Lions \cite{CaLi82} with variational and compactness arguments. A second approach, based on spectral arguments, was introduced by Weinstein \cite{We85,We86} and then considerably generalized by Grillakis, Shatah and Strauss \cite{GrShSt87,GrShSt90} (see also \cite{BiGeNo15}, \cite{BiNoSi19}). In our work, we use the variational techniques to study the stability of standing waves. First, we define
\begin{align*}
S_{\omega}(v)&:=\frac{1}{2}\left[\norm{v_x}^2_{L^2(\R^{+})}+\omega\norm{v}^2_{L^2(\R^{+})}+\alpha|v(0)|^2\right]-\frac{1}{32}\norm{v}^6_{L^6(\R^{+})}, \\
K_{\omega}(v)&:= \norm{v_x}^2_{L^2(\R^{+})}+\omega\norm{v}^2_{L^2(\R^{+})}+\alpha|v(0)|^2 - \frac{3}{16}\norm{v}^6_{L^6(\R^{+})}.
\end{align*} 
We are interested in the following variational problem:
\begin{align}
d(\omega) &:= \inf \left\{S_{\omega}(v) \mid K_{\omega}(v)=0, v \in H^1(\R^{+})\setminus \left\{0\right\}\right\} \; \label{minprob3}.
\end{align} 

We have the following result.
\begin{proposition}\label{variation characteristic problem}
Let $\omega,\alpha \in \R$ such that $\omega>\alpha^2$. All minimizers of \eqref{minprob3} are of form $e^{i\theta}\varphi$, where $\theta \in \R$ and $\varphi$ is given by 
\[
\varphi=2\sqrt[4]{\omega}\operatorname{sech}^{\frac{1}{2}}\left(2\sqrt{\omega}|x|+\tanh^{-1}\left(\frac{-\alpha}{\sqrt{\omega}}\right)\right).
\]  
\end{proposition}   

We give the definition of stability and instability by blow up in $H^1(\R^{+})$. Let $w(t,x)=e^{i\omega t}\varphi(x)$ be a standing wave solution of \eqref{eq first}. 
\begin{itemize}
\item[(1)] The standing wave $w$ is called \emph{orbitally stable} in $H^1(\R^{+})$ if for all $\varepsilon>0$, there exists $\delta>0$ such that if $v_0 \in H^1(\R^{+})$ satisfies 
\[
\norm{v_0-\varphi}_{H^1(\R^{+})} \leq \delta,
\]
then the associated solution $v$ of \eqref{eq first} satisfies 
\[
\mathop{\sup}\limits_{t \in \R}\mathop{\inf}\limits_{\theta \in \R}\norm{v(t)-e^{i\theta}\varphi}_{H^1(\R^{+})} < \varepsilon.
\]  
Otherwise, $w$ said to be \emph{instable}. 
\item[(2)] The standing wave $w$ is called \emph{instable by blow up} if there exists a sequence $(\varphi_n)$ such that $\mathop{\lim}\limits_{n\rightarrow \infty}\norm{\varphi_n-\varphi}_{H^1(\R^{+})} = 0$ and the associated solution $v_n$ of \eqref{eq first} blows up in finite time for all $n$.  
\end{itemize}

Our second main result is the following.

\begin{theorem}\label{stable and instale of SDW for eq 2}
Let $\alpha,\omega\in \R$ be such that $\omega>\alpha^2$. The standing wave $e^{i\omega t}\varphi$, where $\varphi$ is the profile as in Proposition \ref{variation characteristic problem}, solution of \eqref{eq first}, satisfies the following properties.
\begin{itemize}
\item[(1)] If $\alpha<0$ then the standing wave is orbitally stable in $H^1(\R^{+})$.
\item[(2)] If $\alpha>0$ then the standing wave is instable by blow up.
\end{itemize}  
\end{theorem}

\begin{remark}
To our knowledge, the conservation law play an important role to study the stability of standing waves. However, the existence of conservation of energy is not always true (see remark \ref{remarkonconservationlaw}). Our work can only extend for the models with nonlinear terms provide the conservation law of solution.  
\end{remark}

This paper is organized as follows. First, under the assumption of local well posedness in $H^1(\R^{+})$, we prove the existence of blowing up solutions using a virial argument Theorem \ref{thm blow up}. In section \ref{sectionblowup}, we give the proof of Theorem \ref{thm blow up}. Second, in the case $\alpha<0$, using similar arguments as in \cite{CoOh06}, we prove the orbital stability of standing waves of \eqref{eq first}. In the case $\alpha >0$, using similar arguments as in \cite{Stefan09}, we prove the instability by blow up of standing waves. The proof of Theorem \ref{stable and instale of SDW for eq 2} is obtained in Section \ref{sectionstabilty}.       

\subsection*{Acknowledgement}
The author wishes to thank Prof.Stefan Le Coz for his guidance and encouragement.

\section{Proof of the main results}
We consider the equation \eqref{eq first} and assume that the assumption \ref{assumptionB} holds. 

\subsection{The existence of a blowing-up solution}\label{sectionblowup}
In this section, we give the proof of Theorem \ref{thm blow up} using a virial argument (see e.g \cite{Gl77} or \cite{Wu13} for similar arguments). Let $\alpha>0$. Let $v$ be a solution of \eqref{eq first}. 
To prove the existence of blowing up solutions we use similar arguments as in \cite{Wu13}. Set
\[
I(t) = \int_0^{\infty}x^2|v(t)|^2 \, dx.
\]
Let 
\begin{equation}\label{eq alpha}
u(t,x)= v(t,x) \operatorname{exp}\left(-\frac{i}{4}\int_x^{+\infty} |v|^2 \, dy \right)
\end{equation}
be a Gauge transform in $H^1(\R_{+})$. Then the problem \eqref{eq first} is equivalent with
\begin{equation}
\begin{cases}\label{eq 3.1}
iu_t+u_{xx}=i|u|^2u_x ,\\
u_x(0)=\alpha u(0)+\frac{i}{4}|u(0)|^2u(0).
\end{cases}
\end{equation}
The equation \eqref{eq 3.1} has a simpler nonlinear form, but we pay this simplification with a nonlinear boundary condition. Observe that 
\[
I(t)=\int_0^{\infty}x^2|u(t)|^2 \, dx =\int_0^{\infty}x^2|v(t)|^2\, dx.
\]
By a direct calculation, we get
\begin{align}\label{eq derivative in time of I}
\partial_t I(t) &=2\Re\int_0^{\infty}x^2\overline{u(t,x)}\partial_tu(t,x)\,dx=2\Re\int_0^{\infty}x^2\overline{u}(iu_{xx}+|u|^2u_x)\,dx\\
&=2\Im\int_0^{\infty}2x\overline{u}u_x\,dx-\frac{1}{2}\int_0^{\infty}2x|u|^4\,dx\\
&=4\Im\int_0^{\infty}xu_x\overline{u}\,dx-\int_0^{\infty}x|u|^4\,dx.
\end{align}
Define 
\[
J(t)=\Im\int_0^{\infty}xu_x\overline{u}\,dx.
\]
We have
\begin{align*}
\partial_tJ(t) &=\int_0^{\infty}xu_x\overline{u}_t\,dx+\int_0^tx\overline{u}u_{xt}\,dx\\
&=-\Im\int_0^{\infty}xu_t\overline{u}_x\,dx-\Im\int_0^{\infty}(x\overline{u})_xu_t\,dx\\
&= -2\Im\int_0^{\infty}xu_t\overline{u}_x\,dx-\Im\int_0^{\infty}u_t\overline{u}\,dx\\
&=-2\Im\int_0^{\infty}x\overline{u}_x(iu_{xx}+|u|^2u_x)\,dx-\Im\int_0^{\infty}\overline{u}(iu_{xx}+|u|^2u_x)\,dx\\
&= -2\Re\int_0^{\infty}x\overline{u}_xu_{xx}\,dx-\Re\int_0^{\infty}\overline{u}u_{xx}\,dx-\Im\int_0^{\infty}|u|^2u_x\overline{u}\,dx \\
&=-\int_0^{\infty}x\partial_x|u_x|^2\,dx-\Re(\overline{u}u_x)(+\infty)+\Re(\overline{u}u_x)(0)+\Re\int_0^{\infty}\overline{u_x}u_x \, dx -\Im\int_0^{\infty}|u|^2u_x\overline{u}\, dx\\
&=\int_0^{\infty}|u_x|^2\,dx+\Re(\overline{u}(0)u_x(0))+\int_0^{\infty}|u_x|^2\,dx-\Im\int_0^{\infty}|u|^2u_x\overline{u}\,dx\\
&=2\int_0^{\infty}|u_x|^2\,dx-\Im\int_0^{\infty}|u|^2u_x\overline{u}\,dx+\Re(\overline{u}(0)u_x(0)).
\end{align*}
Using the Robin boundary condition we have
\[
\partial_tJ(t)=2\int_0^{\infty}|u_x|^2\,dx-\Im\int_0^{\infty}|u|^2u_x\overline{u}\,dx+\alpha|u(0)|^2.
\]
Moreover using the expression of $v$ in term of $u$ given in \eqref{eq alpha}, we get
\begin{align*}
\partial_tJ(t)&=2\int_0^{\infty}|v_x|^2\,dx-\frac{1}{8}\int_0^{\infty}|v|^6\,dx+\alpha|v(0)|^2\\
&=4E(v)-\alpha|v(0)|^2\leq 4E(v)=4E(\varphi).
\end{align*}
By integrating the two sides of the above inequality in time we have
\begin{align}\label{eq estimate of J}
J(t) &\leq J(0)+4E(\varphi) t.
\end{align}
Integrating the two sides of \eqref{eq derivative in time of I} in time we have
\begin{align*}
I(t) & = I(0) + 4\int_0^t J(s)\, ds - \int_0^t\int_0^{\infty}x|u(s,x)|^4\, dx\, ds \\
& \leq I(0)+4\int_0^t J(s)\, ds.
\end{align*}
Using \eqref{eq estimate of J} we have
\begin{align*}
I(t) &\leq I(0)+4\int_0^t (J(0)+4E(\varphi)s) \, ds \\
& \leq I(0)+4J(0)t +8E(\varphi) t^2.  
\end{align*}
From the assumption $E(\varphi)<0$, there exists a finite time $T_{*}>0$ such that $I(T_{*})=0$,
\[
I(t)>0 \text{ for } 0<t<T_{*}.
\] 
Note that
\begin{align*}
\int_0^{\infty}|\varphi(x)|^2 \, dx &=\int_0^{\infty}|v(t,x)|^2\, dx = -2\Re\int_0^{\infty}xv(t,x)\overline{v_x}(t,x)\,dx \\
&\leq 2\norm{xv}_{L^2_x(\R_{+})}\norm{v_x}_{L^2_x(\R_{+})} =2\sqrt{I(t)}\norm{v_x}_{L^2_x(\R_{+})}.
\end{align*}
Then there exists a constant $C=C(\varphi)>0$ such that
\[
\norm{v_x}_{L^2_x(\R_{+})} \geq \frac{C}{2\sqrt{I(t)}} \rightarrow +\infty \text{ as } t \rightarrow T_{*}.
\] 
Then the solution $v$ blows up in finite time in $H^1(\R^{+})$. This complete the proof of Theorem \ref{thm blow up}.

\subsection{Stability and instability of standing waves}
\label{sectionstabilty}
In this section, we give the proof of Theorem \ref{stable and instale of SDW for eq 2}. First, we find the form of the standing waves of \eqref{eq first}.
\subsubsection{Standing waves}

Let $v=e^{i\omega t}\varphi$ be a solution of \eqref{eq first}. Then $\varphi$ solves
\begin{equation}\label{eq sdw new}
\begin{cases}
0 = \varphi_{xx}-\omega\varphi+\frac{1}{2}\Im(\varphi_x\overline{\varphi})\varphi+\frac{3}{16}|\varphi|^4\varphi ,\ \text{ for } x>0 \\
\varphi_x(0)=\alpha\varphi(0),\\
\varphi\in H^2(\R^{+}).
\end{cases}
\end{equation}
Set 
\[
A:= \omega-\frac{1}{2}\Im(\varphi_x\overline{\varphi})-\frac{3}{16}|\varphi|^4
\]
By writing $\varphi=f+ig$ for $f$ and $g$ real valued functions, for $x>0$, we have
\begin{align*}
f_{xx}&=Af ,\\
g_{xx}&=Ag. 
\end{align*}
Thus,
\[
\partial_x(f_xg-g_xf)=f_{xx}g-g_{xx}f=0 \text{ when } x\neq 0.
\]
Hence, by using $f,g \in H^2(\R^{+})$, we have
\[
f_x(x)g(x)-g_x(x)f(x)=0 \text{ when } x\neq 0.
\]
Then, for all $x \neq 0$, we have
\[
\Im(\varphi_x(x)\overline{\varphi (x)})=g_x(x)f(x)-f_x(x)g(x)=0,
\]
hence, \eqref{eq sdw new} is equivalent to
\begin{equation}\label{eq sdw of robbin condition}
\begin{cases}
0=\varphi_{xx}-\omega\varphi+\frac{3}{16}|\varphi|^4\varphi,\ \text{ for } x>0\\
\varphi_x(0)=\alpha\varphi(0),\\
\varphi\in H^2(\R^{+}).
\end{cases}
\end{equation}

We have the following description of the profile $\varphi$.

\begin{proposition}\label{thm standing waves of Robin condition}
Let $\omega>\alpha^2$. There exists a unique (up to phase shift) solution $\varphi$ of \eqref{eq sdw of robbin condition}, which is of the form
\begin{equation}\label{eq sdw robin}
\varphi = 2\sqrt[4]{\omega}\operatorname{sech}^{\frac{1}{2}}\left(2\sqrt{\omega}|x|+\tanh^{-1}\left(\frac{-\alpha}{\sqrt{\omega}}\right)\right),
\end{equation}
for all $x>0$.
\end{proposition}

\begin{proof}
Let $w$ be the even function defined by
\[
w(x)=\left\{\begin{matrix}
\varphi(x) \text{ if } x \geq 0, \\
\varphi(-x) \text{ if } x \leq 0.
\end{matrix}\right.
\]
Then $w$ solves
\begin{equation}
\begin{cases}
0=-w_{xx}+\omega w-\frac{3}{16}|w|^4w, \text{ for } x\neq 0,\\
w_x(0^{+})-w_x(0^{-})=2\alpha w(0),\\
w \in H^2(\R)\setminus \left\{0\right\} \cap H^1(\R).
\end{cases}
\end{equation}
Using the results of Fukuizumi and Jeanjean \cite{FuJe08}, we obtain that 
\[
w(x)=2\sqrt[4]{\omega}\operatorname{sech}^{\frac{1}{2}}\left(2\sqrt{\omega}|x|+\tanh^{-1}\left(\frac{-\alpha}{\sqrt{\omega}}\right)\right)
\]
up to phase shift provided $\omega>\alpha^2$. Hence, for $x>0$ we have
\[
\varphi(x)=2\sqrt[4]{\omega}\operatorname{sech}^{\frac{1}{2}}\left(2\sqrt{\omega}|x|+\tanh^{-1}\left(\frac{-\alpha}{\sqrt{\omega}}\right)\right)
\]
up to phase shift. This implies the desired result.
\end{proof}

\subsubsection{The variational problems}
\label{section variation problem}

In this section, we give the proof of Proposition \ref{variation characteristic problem}.

First, we introduce another variational problem:
\begin{align}
\widetilde{d}(\omega)&:= \inf \left\{\widetilde{S}_{\omega}(v) \mid v \text{ even}, \widetilde{K}_{\omega}(v)=0, v \in H^1(\R)\setminus \left\{0\right\}\right\} , \; \label{minprob1}
\end{align}
where $\widetilde{S}_{\omega}$, $\widetilde{K}_{\omega}$ are defined for all $v\in H^1(\R)$ by
\begin{align*}
\widetilde{S}_{\omega}(v)&:=\frac{1}{2}\left[\norm{v_x}^2_{L^2(\R)}+\omega\norm{v}^2_{L^2(\R)}+2\alpha|v(0)|^2\right]-\frac{1}{32}\norm{v}^6_{L^6(\R)},\\
\widetilde{K}_{\omega}(v)&:=\norm{v_x}^2_{L^2(\R)}+\omega\norm{v}^2_{L^2(\R)}+2\alpha|v(0)|^2-\frac{3}{16}\norm{v}^6_{L^6(\R)}. 
\end{align*}
The functional $\widetilde{K}_{\omega}$ is called Nehari functional. The following result has proved in \cite{FuJe08,FuOhOz08}.  
\begin{proposition}
\label{new proposition 1}
Let $\omega>\alpha^2$ and $\varphi$ satisfies
\begin{equation}\label{3}
\begin{cases}
-\varphi_{xx}+ 2\alpha\delta\varphi+ \omega\varphi-\frac{3}{16}|\varphi|^4\varphi=0,\\
\varphi \in H^1(\R) \setminus\left\{0\right\}.
\end{cases}
\end{equation}
Then, there exists a unique positive solution $\varphi$ of \eqref{3}. This solution is the unique positive minimizer of \eqref{minprob1}. Furthermore, we have an explicit formula for $\varphi$
\[
\varphi (x)=2\sqrt[4]{\omega}\operatorname{sech}^{\frac{1}{2}}\left(2\sqrt{\omega}|x|+\tanh^{-1}\left(\frac{-\alpha}{\sqrt{\omega}}\right)\right).
\]
\end{proposition}

We have the following relation between the variational problems.
\begin{proposition}
\label{proposition vartiational 1}
Let $\omega>\alpha^2$. We have 
\[
d(\omega)=\frac{1}{2}\widetilde{d}(\omega).
\]
\end{proposition}
\begin{proof}
Assume $v$ is a minimizer of \eqref{minprob3}, define the $H^1(\R)$ function $w$ by
\[
w(x)=\left\{ 
\begin{matrix}
v(x) \text{ if } x>0,\\
v(-x) \text{ if } x<0.
\end{matrix}
\right.
\]    
The function $w \in H^1(\R) \setminus \left\{0\right\}$ verifies
\begin{align*}
\widetilde{S}_{\omega}(w) &= 2 S_{\omega}(v)=2d(\omega),\\
\widetilde{K}_{\omega}(w) &= 2 K_{\omega}(v)=0.
\end{align*}
This implies that
\begin{equation}\label{relation between minima 1}
\widetilde{d}(\omega) \leq \widetilde{S}_{\omega}(w) =2d(\omega). 
\end{equation}
Now, assume $v$ is a minimizer of \eqref{minprob1}. Let $w$ be the restriction of $v$ on $\R^{+}$, then, 
\[
K_{\omega}(w)=\frac{1}{2}\widetilde{K}_{\omega}(v)=0.
\]
Hence, we obtain
\begin{equation}\label{relation between minima 2}
\widetilde{d}(\omega)=\widetilde{S}_{\omega}(v)=2S_{\omega}(w)\geq 2d(\omega).  
\end{equation}
Combining \eqref{relation between minima 1} and \eqref{relation between minima 2} we have 
\[
\widetilde{d}(\omega)=2d(\omega).
\]
This implies the desired result.
\end{proof} 
\begin{proof}[Proof of Theorem \ref{variation characteristic problem}] 
Let $v$ be a minimizer of \eqref{minprob3}. Define $w(x) \in H^1(\R)$ by
\[
w(x)=\left\{
\begin{matrix}
v(x) \text{ if } x>0,\\
v(-x) \text{ if } x<0.
\end{matrix}
\right.
\]
Then, $w$ is an even function. Moreover, $w$ satisfies 
\begin{align*}
\widetilde{K}_{\omega}(w)&=2K_{\omega}(v)=0,\\
\widetilde{S}_{\omega}(w)&=2S_{\omega}(v)=2d(\omega)=\widetilde{d}(\omega). 
\end{align*}
Hence, $w$ is a minimizer of \eqref{minprob1}. From Propositions \ref{new proposition 1}, \ref{proposition vartiational 1}, $w$ is of the form $e^{i\theta}\varphi$, where  $\theta \in \R$ is a constant and $\varphi$ is of the form
\[
2\sqrt[4]{\omega}\operatorname{sech}^{\frac{1}{2}}\left(2\sqrt{\omega}|x|+\tanh^{-1}\left(\frac{-\alpha}{\sqrt{\omega}}\right)\right).
\]
Hence, $v=w\vert_{\R^{+}}$ satisfies
\[
v(x) = e^{i\theta}\varphi(x), 
\]
for $x>0$. This completes the proof of Proposition \ref{variation characteristic problem}.
\end{proof}


\subsubsection{Stability and instability of standing waves}

In this section, we give the proof of Theorem \ref{stable and instale of SDW for eq 2}. We use the notations $\widetilde{S}_{\omega}$ and $\widetilde{K}_{\omega}$ as in Section \ref{section variation problem}. First, we define
\begin{align}
N(v)&:= \frac{3}{16}\norm{v}^6_{L^6(\R^{+})}, \label{N}\\
L(v)&:= \norm{v_x}^2_{L^2(\R^{+})}+\omega\norm{v}^2_{L^2(\R^{+})}+\alpha|v(0)|^2. \label{L}
\end{align}
We can rewrite $S_{\omega},K_{\omega}$ as follows
\begin{align*}
S_{\omega}&= \frac{1}{2}L-\frac{1}{6}N,\\
K_{\omega}&= L-N.
\end{align*}
We have the following classical properties of the above functions.
\begin{lemma}\label{lm property}
Let $(\omega,\alpha)\in \R^2$ such that $\omega>\alpha^2$. The following assertions hold.
\begin{itemize}
\item[(1)] There exists a constant $C>0$ such that 
\[
L(v) \geq C\norm{v}^2_{H^1(\R^{+})} \quad \forall v \in H^1(\R^{+}).
\]
\item[(2)] We have $d(\omega)>0$.
\item[(3)] If $v\in H^1(\R^{+})$ satisfies $K_{\omega}(v)<0$ then $L(v)>3d(\omega)$. 
\end{itemize}
\end{lemma}
\begin{proof}
We have
\begin{align*}
|v(0)|^2 &= -\int_0^{\infty}\partial_x(|v(x)|^2)\,dx=-2\Re\int_0^{\infty}v(x)\overline{v}_x(x)\, dx\\
&\leq 2\norm{v}_{L^2(\R^{+})}\norm{v_x}_{L^2(\R^{+})}.
\end{align*}
Hence,
\begin{align*}
L(v)&=\norm{v_x}^2_{L^2(\R^{+})}+\omega\norm{v}^2_{L^2(\R^{+})}+\alpha|v(0)|^2\\
&\geq \norm{v_x}^2_{L^2(\R^{+})}+\omega\norm{v}^2_{L^2(\R^{+})}-2|\alpha|\norm{v}_{L^2(\R^{+})}\norm{v_x}_{L^2(\R^{+})}\\
&\geq C\norm{v}^2_{H^1(\R^{+})}+(1-C)\norm{v_x}^2_{L^2(\R^{+})}+(\omega-C)\norm{v}^2_{L^2(\R^{+})}-2|\alpha|\norm{v}_{L^2(\R^{+})}\norm{v_x}_{L^2(\R^{+})}\\
&\geq C\norm{v}^2_{H^1(\R^{+})}+(2\sqrt{(1-C)(\omega-C)}-2|\alpha|)\norm{v}_{L^2(\R^{+})}\norm{v_x}_{L^2(\R^{+})}.
\end{align*}
From the assumption $\omega>\alpha^2$, we can choose $C\in (0,1)$ such that
\[
2\sqrt{(1-C)(\omega-C)}-2|\alpha|>0.
\] 
This implies (1). Now, we prove (2). Let $v$ be an element of $H^1(\R^{+})$ satisfying $K_{\omega}(v)=0$. We have
\[
C\norm{v}^2_{H^1(\R^{+})} \leq L(v) = N(v) \leq C_1 \norm{v}^6_{H^1(\R^{+})}.
\]
Then, 
\[
\norm{v}^2_{H^1(\R^{+})} \geq \sqrt[4]{\frac{C}{C_1}}.
\]
From the fact that, for $v$ satisfying $K_{\omega}(v)=0$, we have $S_{\omega}(v)=S_{\omega}(v)-\frac{1}{6}K_{\omega}(v)=\frac{1}{3}L(v)$, this implies that
\[
d(\omega) = \frac{1}{3}\inf\left\{L(v): v \in H^1(\R^{+}), K_{\omega}(v)=0\right\} \geq \frac{C}{3}\sqrt[4]{\frac{C}{C_1}} >0.
\]
Finally, we prove (3). Let $v \in H^1(\R^{+})$ satisfying $K_{\omega}(v)< 0$. Then, there exists $\lambda_1 \in (0,1)$ such that $K_{\omega}(\lambda_1v)=\lambda_1^2L(v)-\lambda_1^6N(v)=0$. Since $v \neq 0$, we have $3d(\omega)\leq L(\lambda_1 v)=\lambda_1^2 L(v)< L(v)$.
\end{proof}

Define
\begin{align}
\tilde N(v)&:= \frac{3}{16}\norm{v}^6_{L^6}, \label{N}\\
\tilde L(v)&:= \norm{v_x}^2_{L^2}+\omega\norm{v}^2_{L^2}+2\alpha|v(0)|^2. \label{L}
\end{align}
We can rewrite $S_{\omega},K_{\omega}$ as follows
\begin{align*}
\tilde{S}_{\omega}&= \frac{1}{2}\tilde L-\frac{1}{6}\tilde N,\\
\tilde{K}_{\omega}&= \tilde L- \tilde N.
\end{align*}
As consequence of the previous lemma, we have the following result.

\begin{lemma}\label{lm property}
Let $(\omega,\alpha)\in \R^2$ such that $\omega>\alpha^2$. The following assertions hold.
\begin{itemize}
\item[(1)] There exists a constant $C>0$ such that 
\[
\tilde L(v) \geq C\norm{v}^2_{H^1} \quad \forall v \in H^1(\R).
\]
\item[(2)] We have $\tilde d(\omega)>0$.
\item[(3)] If $v\in H^1$ satisfies $\tilde{K}_{\omega}(v)<0$ then $\tilde L(v)>3 \tilde d(\omega)$. 
\end{itemize}
\end{lemma}

We introduce the following properties.
\begin{lemma}\label{lmneedtoprovecompactnessinR}
 Let $2 \leq p <\infty$ and $(f_n)$ be a bounded sequence in $L^p(\R)$. Assume that $f_n \rightarrow f$ a.e in $\R$. Then we have
 \[
 \norm{f_n}^p_{L^p}-\norm{f_n-f}^p_{L^p}-\norm{f}^p_{L^p} \rightarrow 0.
 \]
\end{lemma}

\begin{lemma}
The following minimization problem is equivalent to the problem \eqref{minprob1} i.e same minimum and the minimizers:
\begin{equation}
\label{minimizernew}
d:= \inf\left\{\frac{1}{16}\norm{u}^6_{L^6}: u \text{ even }, u \in H^1(\R)\setminus \{0\}, \tilde{K}_{\omega}(u) \leq 0 \right\}.
\end{equation} 
\end{lemma}

\begin{proof}
We see that the minimizer problem \eqref{minprob1} is equivalent to following problem:
\begin{equation}\label{minnew2}
\inf\left\{\frac{1}{16}\norm{u}^6_{L^6}:u \text{ even } u \in H^1(\R)\setminus \{0\}, \tilde{K}_{\omega}(u) = 0 \right\}.
\end{equation}
Let $v$ be a minimizer of \eqref{minprob1} then $\tilde{K}_{\omega}(v) \leq 0$, hence, $\tilde d(\omega)=\frac{1}{16}\norm{v}^6_{L^6} \geq d$. Now, let $v$ be a minimizer of \eqref{minimizernew}. We prove that $\tilde{K}_{\omega}(v)=0$. Indeed, assuming $\tilde{K}_{\omega}(v) <0$, we have
\[
\tilde{K}_{\omega}(\lambda v) = \lambda^2\left(\norm{v_x}^2_{L^2}+\omega \norm{v}^2_{L^2}+2\alpha |v(0)|^2 - \frac{3\lambda^4}{16}\norm{v}^6_{L^6}\right) \leq 0,
\]  
as $0<\lambda$ is small enough. Thus, by continuity, there exists a $\lambda_0 \in (0,1)$ such that $\tilde{K}_{\omega}(\lambda_0 v)=0$. We have $d<\tilde d(\omega) \leq \frac{1}{16}\norm{\lambda_0 v}^6_{L^6} < \frac{1}{16}\norm{v}^6_{L^6}=d$. Which is a contradiction. It implies that $\tilde{K}_{\omega}(v)=0$ and $v$ is a minimizer of \eqref{minnew2}, hence $v$ is a minimizer of \eqref{minprob1}. This completes the proof.  
\end{proof}
Now, using the similar arguments in \cite[Proof of Proposition 2]{FuOhOz08}, we have the following result.
\begin{proposition}\label{procompactness property}  
Let $(\omega,\alpha) \in \R^2$ be such that $\alpha<0$, $\omega>\alpha^2$ and $(w_n) \subset H^1(\R)$ be a even sequence satisfying the following properties
\begin{align*}
\widetilde{S}_{\omega}(w_n) &\rightarrow \tilde{d}(\omega),\\
\widetilde{K}_{\omega}(w_n) &\rightarrow 0.
\end{align*}

as $n \rightarrow \infty$. Then, there exists a minimizer $w$ of \eqref{minprob1} such that $w_n \rightarrow w$ strongly in $H^1(\R)$ up to subsequence.   
\end{proposition}
\begin{proof}
In what follows, we shall often extract subsequence without mentioning this fact explicitly. We divide the proof into two steps.

\textbf{Step 1. Weakly convergence to a nonvanishing function of minimizer sequence} We have
\[
\frac{1}{3}\tilde L(w_n)=\tilde{S}_{\omega}(w_n)-\frac{1}{6}\tilde{K}_{\omega}(w_n) \rightarrow \tilde d(\omega),
\] 
as $n \rightarrow \infty$. Then, $(w_n)$ is bounded in $H^1(\R)$ and there exists $w \in H^1(\R)$ even such that $w_n \rightharpoonup w$ in $H^1(\R)$ up to subsequence. We prove $w \neq 0$. Assume that $w \equiv 0$. Define, for $u \in H^1(\R)$,
 \begin{align*}
 S^0_{\omega}(u)&=\frac{1}{2}\norm{u_x}^2_{L^2}+\frac{\omega}{2}\norm{u}^2_{L^2} -\frac{1}{32}\norm{u}^6_{L^6},\\
 K^0_{\omega}(u)&=\norm{u_x}^2_{L^2}+\omega\norm{u}^2_{L^2}-\frac{3}{16}\norm{u}^6_{L^6}.  
 \end{align*}
 Let $\psi_{\omega}$ be minimizer of following problem
 \begin{align*}
 d^0(\omega)&=\inf\left\{S^0_{\omega}(u): u \text{ even }, u \in H^1(\R) \setminus \{0\}, K^0_{\omega}(u) = 0\right\}\\
 &=\inf\left\{\frac{1}{16}\norm{u}^6_{L^6}: u \text{ even }, u \in H^1(\R) \setminus \{0\}, K^0_{\omega}(u) \leq 0\right\}.
 \end{align*}
We have $K^0_{\omega}(w_n)=\tilde{K}_{\omega}(w_n)-2\alpha |w_n(0)|^2 \rightarrow 0$, as $n \rightarrow \infty$. Since, $\alpha<0$. we have $\tilde{K}_{\omega}(\psi_{\omega})<0$ and hence we obtain
\begin{equation}
\label{estimaterelationofminimum}
\tilde d(\omega) < \frac{1}{16}\norm{\psi_{\omega}}^6_{L^6} = d^0(\omega)
\end{equation}
We set 
\[
\lambda_n=\left(\frac{\norm{\partial_x w_n}^2_{L^2}+\omega \norm{w_n}^2_{L^2}}{\frac{3}{16}\norm{w_n}^6_{L^6}}\right)^{\frac{1}{4}}.
\]
We here remark that $0<\tilde d(\omega)=\mathop{\lim}\limits_{n\rightarrow \infty}\frac{1}{16}\norm{w_n}^6_{L^6}$. It follows that
\[
\lambda_n^4-1 = \frac{K^0_{\omega}(w_n)}{\frac{3}{16}\norm{w_n}^6_{L^6}} \rightarrow 0,
\]
as $n \rightarrow \infty$. We see that $K^0_{\omega}(\lambda_n w_n)=0$ and $\lambda_n w_n \neq 0$. By the definition of $d^0(\omega)$, we have
\[
d^0(\omega)\leq \frac{1}{16}\norm{\lambda_n w_n}^6_{L^6} \rightarrow \tilde d(\omega) \text{ as } n \rightarrow \infty.
\] 
This contradicts to \eqref{estimaterelationofminimum}. Thus, $w \neq 0$.

\textbf{Step 2. Conclude the proof} Using Lemma \ref{lmneedtoprovecompactnessinR} we have
\begin{align}
&\tilde{K}_{\omega}(w_n)-\tilde{K}_{\omega}(w_n-w)-\tilde{K}_{\omega}(w)  \rightarrow 0, \label{convergence by lemma 1} \\
&\tilde L(w_n)-\tilde L(w_n-w)-\tilde L(w) \rightarrow 0. \label{convergence by lemma 2}
\end{align}
Now, we prove $\tilde{K}_{\omega}(w) \leq 0$ by contradiction. Suppose that $\tilde{K}_{\omega}(w) >0$. By  the assumption $\tilde{K}_{\omega}(w_n) \rightarrow 0$ and \eqref{convergence by lemma 1}, we have
\[
\tilde{K}_{\omega}(w_n-w) \rightarrow -\tilde{K}_{\omega}(w) <0.
\]
Thus, $\tilde{K}_{\omega}(w_n-w)<0$ for $n$ large enough. By Lemma \ref{lm property} (3), we have $\tilde L(w_n-w) \geq 3\tilde d(\omega)$. Since $\tilde L(w_n) \rightarrow 3\tilde d(\omega)$, by \eqref{convergence by lemma 2}, we have
\[
\tilde L(w) = \mathop{\lim}\limits_{n\rightarrow\infty}(\tilde L(w_n)-\tilde L(w_n-w)) \leq 0.
\] 
Moreover, $w \neq 0$ and by Lemma \ref{lm property} (1), we have $\tilde L(w)>0$. This is a contradiction. Hence, $\tilde{K}_{\omega}(w)<0$. By Lemma \ref{lm property} (2), (3) and weakly lower semicontinuity of $\tilde L$, we have 
\[
3\tilde d(\omega)\leq \tilde L(w) \leq \mathop{\lim}\limits_{n\rightarrow\infty}\inf \tilde L(w_n)=3\tilde d(\omega).
\]
Thus, $\tilde L(w)=3\tilde d(\omega)$. Combining with \eqref{convergence by lemma 2}, we have $\tilde L(w_n-w) \rightarrow 0$, as $n \rightarrow \infty$. By Lemma \ref{lm property} (1), we have $w_n \rightarrow w$ strongly in $H^1(\R)$. Hence, $w$ is a minimizer of \eqref{minprob1}. This completes the proof. 
\end{proof}

To prove the stability statement (1) for $\alpha<0$ in Theorem \ref{stable and instale of SDW for eq 2}, we will use similar arguments as in the work of Colin and Ohta \cite{CoOh06}. We need the following property.
\begin{lemma}\label{convegence to grouns state}
Let $\alpha<0$, $\omega >\alpha^2$. If a sequence $(v_n) \subset H^1(\R^{+})$ satisfies 
\begin{align}
S_{\omega}(v_n) & \rightarrow d(\omega),\label{assumption1} \\ 
K_{\omega}(v_n) &\rightarrow 0,\label{assumption}
\end{align}
then there exist a constant $\theta_0 \in \R$ such that  $v_n\rightarrow e^{i\theta_0}\varphi$, up to subsequence, where $\varphi$ is defined as in Proposition \ref{variation characteristic problem}.
\end{lemma}  

\begin{proof}
Define the sequence $(w_n) \subset H^1(\R)$ as follows,
\[
w_n(x)=\left\{\begin{matrix}
v_n(x) \text{ for } x>0 ,\\ v_n(-x) \text{ for } x<0.
\end{matrix}\right.
\]
We can check that
\begin{align*}
\widetilde{S}_{\omega}(w_n)=2S_{\omega}(v_n) &\rightarrow 2d(\omega)=\tilde{d}(\omega),\\
\widetilde{K}_{\omega}(w_n)=2K_{\omega}(v_n) &\rightarrow 0,
\end{align*}
as $n \rightarrow \infty$. Using Proposition \ref{procompactness property}, there exists a minimizer $w_0$ of \eqref{minprob1} such that $w_n \rightarrow w_0$ strongly in $H^1(\R)$, up to subsequence. For convenience, we assume that $w_n \rightarrow w_0$ strongly in $H^1(\R)$. By Proposition \ref{new proposition 1}, there exists a constant $\theta_0\in \R$ such that 
\[
w_0=e^{i\theta_0}\tilde{\varphi},
\]
where $\tilde{\varphi}$ is defined by
\begin{align}\label{111}
\tilde{\varphi}(x)=\left\{\begin{matrix}
\varphi(x) \text{ for } x>0 ,\\ \varphi(-x) \text{ for } x<0.
\end{matrix}\right.
\end{align}
Hence, the sequence $(v_n)$ is the restriction of the sequence $(w_n)$ on $\R^{+}$, and satisfies 
\[
v_n \rightarrow e^{i\theta_0}\varphi, \text{ strongly in } H^1(\R^{+}),
\]  
up to subsequence. This completes the proof.
\end{proof}

Define 
\begin{align*}
\mathcal{A}^{+}_{\omega} &= \left\{v \in H^1(\R^{+})\setminus \left\{0\right\} : S_{\omega}(v)<d(\omega),K_{\omega}(v)>0\right\},\\
\mathcal{A}^{-}_{\omega} &= \left\{v \in H^1(\R^{+})\setminus \left\{0\right\} : S_{\omega}(v)<d(\omega),K_{\omega}(v)<0\right\},\\
\mathcal{B}^{+}_{\omega} &= \left\{v \in H^1(\R^{+})\setminus \left\{0\right\} : S_{\omega}(v)<d(\omega),N(v)<3d(\omega)\right\},\\
\mathcal{B}^{-}_{\omega} &= \left\{v \in H^1(\R^{+})\setminus \left\{0\right\} : S_{\omega}(v)<d(\omega),N(v)>3d(\omega)\right\}.
\end{align*}
We have the following result.
\begin{lemma}\label{lemmainvariant under flow}
Let $\omega,\alpha\in \R^2$ such that $\alpha<0$ and $\omega>\alpha^2$.
\begin{itemize}
\item[(1)] The sets $\mathcal{A}^{+}_{\omega}$ and $\mathcal{A}^{-}_{\omega}$ are invariant under the flow of \eqref{eq first}.
\item[(2)] $\mathcal{A}^{+}_{\omega}=\mathcal{B}^{+}_{\omega}$ and $\mathcal{A}^{-}_{\omega}=\mathcal{B}^{-}_{\omega}$.
\end{itemize}
\end{lemma}
\begin{proof}
(1) Let $u_0 \in \mathcal{A}^{+}_{\omega}$ and $u(t)$ the associated solution for \eqref{eq first} on $(T_{min},T_{max})$. By $u_0 \neq 0$ and the conservation laws, we see that $S_{\omega}(u(t)) = S_{\omega}(u_0)<d(\omega)$ for $t \in (T_{min},T_{max})$. Moreover, by definition of $d(\omega)$ we have $K_{\omega}(u(t))\neq 0$ on $(T_{min},T_{max})$. Since the function $t \mapsto K_{\omega}(u(t))$ is continuous, we have $K_{\omega}(u(t))>0$ on $(T_{min},T_{max})$. Hence, $\mathcal{A}^{+}_{\omega}$ is invariant under flow of \eqref{eq first}. By the same way, $\mathcal{A}^{-}_{\omega}$ is invariant under flow of \eqref{eq first}.\\
(2) If $v \in \mathcal{A}^{+}_{\omega}$ then by \eqref{L}, \eqref{N} we have $N(v)=3S_{\omega}(v)-2K_{\omega}(v) <3d(\omega)$, which shows $v \in \mathcal{B}^{+}_{\omega}$, hence $\mathcal{A}^{+}_{\omega} \subset \mathcal{B}^{+}_{\omega}$. Now, let $v \in \mathcal{B}^{+}_{\omega}$. We show $K_{\omega}(v)>0$ by contradiction. Suppose that $K_{\omega}(v)\leq 0$. Then, by Lemma \ref{lm property} (3), $L(v) \geq 3d(\omega)$. Thus, by \eqref{L} and \eqref{N}, we have 
\[
S_{\omega}(v) = \frac{1}{2}L(v)-\frac{1}{6}N(v) \geq d(\omega),
\]        
which contradicts $S_{\omega}(v)<d(\omega)$. Therefore, we have $K_{\omega}(v)>0$, which shows $v \in \mathcal{A}^{+}_{\omega}$ and $\mathcal{B}^{+}_{\omega} \subset \mathcal{A}^{+}_{\omega}$. Next, if $v \in \mathcal{A}^{-}_{\omega}$, then by Lemma \ref{lm property} (3), $L(v) > 3d(\omega)$. Thus, by \eqref{L} and \eqref{N}, we have $N(v)=L(v)-K_{\omega}(v)>3d(\omega)$, which shows $v \in \mathcal{B}^{-}_{\omega}$. Thus, $\mathcal{A}^{-}_{\omega} \subset \mathcal{B}^{-}_{\omega}$. Finally, if $v \in \mathcal{B}^{-}_{\omega}$, then by \eqref{L} and \eqref{N}, we have $2K_{\omega}(v)=3S_{\omega}(v)-N(v) < 3d(\omega)-3d(\omega)=0$, which shows $v \in \mathcal{A}^{-}_{\omega}$, hence, $\mathcal{B}^{-}_{\omega} \subset \mathcal{A}^{-}_{\omega}$. This completes the proof.    
\end{proof}
From Proposition \ref{variation characteristic problem}, we have
\[
d(\omega)=S_{\omega}(\varphi).
\] 
Since $\alpha<0$, we see that
\[
d''(\omega)=\partial_{\omega}\norm{\varphi}^2_{L^2(\R^{+})}=\frac{1}{2}\partial_{\omega}\norm{\tilde{\varphi}}^2_{L^2(\R)}>0,
\] 
where $\tilde{\varphi}$ is defined as \eqref{111} and we know from \cite{FuOhOz08}, \cite{FuJe08} that 
\[
\partial_{\omega}\norm{\tilde{\varphi}}^2_{L^2(\R)}>0,
\]
for $\alpha<0$. We define the function $h : (-\varepsilon_0,\varepsilon_0) \rightarrow \R$ by
\[
h(\tau)=d(\omega\pm\tau),
\] 
for $\varepsilon_0>0$ sufficiently small such that $h''(\tau)>0$ and the sign $+$ or $-$ is selected such that $h'(\tau)>0$ for $\tau \in (-\varepsilon_0,\varepsilon_0)$. Without loss of generality, we can assume 
\[
h(\tau)=d(\omega+\tau).
\] 
\begin{lemma}\label{lm property of h}
Let $(\omega,\alpha) \in \R^2$ such that $\omega>\alpha^2$ and let $h$ be defined as above. Then, for any $\varepsilon \in (0,\varepsilon_0)$, there exists $\delta>0$ such that if $v_0 \in H^1(\R^{+})$ satisfies $\norm{v_0-\varphi}_{H^1(\R^{+})}<\delta$, then the solution $v$ of \eqref{eq first} with $v(0)=v_0$ satisfies $3h(-\varepsilon)<N(v(t))<3h(\varepsilon)$ for all $t \in (T_{min},T_{max})$.   
\end{lemma}

\begin{proof}
The proof of the above lemma is similar to the one of \cite{CoOh06} or \cite{Sh83}. Let $\varepsilon \in (0, \varepsilon_0)$. Since $h$ is increasing, we have $h(-\varepsilon)<h(0)<h(\varepsilon)$. Moreover, by $K_{\omega}(\varphi)=0$ and \eqref{N}, \eqref{L}, we see that $3h(0)=3d(\omega)=3S_{\omega}(\varphi)=N(\varphi)$. Thus, if $u_0\in 
H^1(\R^{+})$ satisfies $\norm{u_0-\varphi}_{H^1(\R^{+})}<\delta$ then we have $3h(0)=N(u_0)+O(\delta)$ and $3h(-\varepsilon)<N(u_0)<3h(\varepsilon)$ for sufficiently small $\delta>0$. Since $h(\pm \varepsilon)=d(\omega \pm \varepsilon)$ and the set $\mathcal{B}^{\pm}_{\omega}$ are invariant under the flow of \eqref{eq first} by Lemma \ref{lemmainvariant under flow}, to conclude the proof, we only have to show that there exists $\delta>0$ such that if $u_0 \in H^1(\R^{+})$ satisfies $\norm{u_0-\varphi}_{H^1(\R^{+})}<\delta$ then $S_{\omega\pm \varepsilon}(u_0)<h(\pm \varepsilon)$. Assume that $u_0\in H^1(\R^{+})$ satisfies $\norm{u_0-\varphi}_{H^1(\R^{+})}<\delta$. We have
\begin{align*}
S_{\omega\pm \varepsilon}(u_0)&=S_{\omega\pm\varepsilon}(\varphi)+O(\delta)\\
&= S_{\omega}(\varphi)\pm \varepsilon M(\varphi)+O(\delta)\\
&=h(0)\pm \varepsilon h'(0)+O(\delta).  
\end{align*}
On the other hand, by the Taylor expansion, there exists $\tau_1=\tau_1(\varepsilon)\in (-\varepsilon_0,\varepsilon_0)$ such that
\[
h(\pm\varepsilon)=h(0)\pm\varepsilon h'(0)+\frac{\varepsilon^2}{2}h''(\tau_1).
\]
Since $h''(\tau_1)>0$ by definition of $h$, we see that there exists $\delta>0$ such that if $u_0\in H^1(\R^{+})$ satisfies $\norm{u_0-\varphi}_{H^1(\R^{+})}<\delta$ then $S_{\omega\pm \varepsilon}(u_0)<h(\pm\varepsilon)$. This completes the proof.
\end{proof}

\begin{proof}[Proof of Theorem \ref{stable and instale of SDW for eq 2} (1)]
Assume that $e^{i\omega t}\varphi$ is not stable for \eqref{eq first}. Then, there exists a constant $\varepsilon_1>0$, a sequence of solutions $(v^n)$ to \eqref{eq first}, and a sequence $\left\{t_n\right\} \in (0,\infty)$ such that
\begin{equation} 
v_n(0) \rightarrow \varphi \text{ in } H^1(\R^{+}),\\
\mathop{\inf}\limits_{\theta\in \R}\norm{v_n(t_n)-e^{i\theta}\varphi}_{H^1(\R^{+})}\geq \varepsilon_1. \label{eq 4.5}
\end{equation}
By using the conservation laws of solutions of \eqref{eq first}, we have
\begin{equation}\label{eq 4.6}
S_{\omega}(v_n(t_n))=S_{\omega}(v_n(0))\rightarrow S_{\omega}(\varphi)=d(\omega).
\end{equation}
Using Lemma \ref{lm property of h}, we have
\begin{equation}\label{eq 4.7}
N(v_n(t_n)) \rightarrow 3d(\omega). 
\end{equation}
Combined \eqref{eq 4.6} and \eqref{eq 4.7}, we have
\[
K_{\omega}(v_n(t_n))=2S_{\omega}(v_n(t_n))-\frac{2}{3}N(v_n(t_n))\rightarrow 0.
\]
Therefore, using Lemma \ref{convegence to grouns state}, there exists $\theta_0 \in \R$ such that $(v_n(t_n,.))$ has a subsequence (we denote it by the same letter) that converges to $e^{i\theta_0}\varphi$ in $H^1(\R^{+})$, where $\varphi$ is defined as in Proposition \ref{variation characteristic problem}. Hence, we have
\begin{equation}
\mathop{\inf}\limits_{\theta \in\R}\norm{v_n(t_n)-e^{i\theta}\varphi}_{H^1(\R^{+})} \rightarrow 0, 
\end{equation} 
as $n \rightarrow \infty$, this contradicts \eqref{eq 4.5}. Hence, we obtain the desired result.
\end{proof}
Next, we give the proof of Theorem \ref{stable and instale of SDW for eq 2} (2), using similar arguments as in \cite{Stefan09}.

Assume $\alpha>0$. Let $e^{i\omega t}\varphi$ be the standing wave solution of \eqref{eq first}. Introduce the scaling
\[
v_{\lambda}(x)=\lambda^{\frac{1}{2}}v(\lambda x).
\]
Let $S_{\omega}$, $K_{\omega}$ be defined as in Proposition \ref{variation characteristic problem}, for convenience, we will remove the index $\omega$. Define
\begin{align*}
P(v)&:=\frac{\partial}{\partial\lambda}S(v_{\lambda})\vert_{\lambda=1}=\norm{v_x}^2_{L^2(\R^{+})}-\frac{1}{16}\norm{v}^6_{L^6(\R^{+})}+\frac{\alpha}{2}|v(0)|^2.
\end{align*}
In the following lemma, we investigate the behaviour of the above functional under scaling.
\begin{lemma}
\label{property of function under scaling 1}
Let $v \in H^1(\R^{+})\setminus \left\{0\right\}$ be such that $v(0) \neq 0$, $P(v)\leq 0$. Then there exists $\lambda_0 \in (0,1]$ such that
\begin{itemize}
\item[(i)] $P(v_{\lambda_0})=0$,
\item[(ii)] $\lambda_0=1$ if only if $P(v)=0$,
\item[(iii)] $\frac{\partial}{\partial\lambda}S(v_{\lambda})=\frac{1}{\lambda}P(v_{\lambda})$,
\item[(iv)]$\frac{\partial}{\partial\lambda}S(v_{\lambda})>0$ on $(0,\lambda_0)$ and $\frac{\partial}{\partial\lambda}S(v_{\lambda})<0$ on $(\lambda_0,\infty)$,
\item[(v)] The function $\lambda \rightarrow S(v_{\lambda})$ is concave on $(\lambda_0,\infty)$. 
\end{itemize}
\end{lemma}

\begin{proof}
A simple calculation leads to 
\[
P(v_{\lambda})=\lambda^2\norm{v_x}^2_{L^2(\R^{+})}-\frac{\lambda^2}{16}\norm{v}^6_{L^6(\R^{+})}+\frac{\lambda\alpha}{2}|v(0)|^2.
\]
Then, for $\lambda>0$ small enough, we have
\[
P(v_{\lambda})>0.
\]
By continuity of $P$, there exists $\lambda_0\in (0,1]$ such that $P(v_{\lambda_0})=0$. Hence (i) is proved. If $\lambda_0=1$ then $P(v)=1$. Conversely, if $P(v)=0$ then
\[
0=P(v_{\lambda_0})=\lambda_0^2P(v)+\frac{\lambda_0-\lambda_0^2}{2}\alpha|v(0)|^2=\frac{\lambda_0-\lambda_0^2}{2}\alpha|v(0)|^2.
\]
By the assumption $v(0) \neq 0$, we have $\lambda_0=1$, hence (ii) is proved. Item (iii) is obtained by a simple calculation. To obtain (iv), we use (iii). We have
\begin{align*}
P(v_{\lambda})&= \lambda^2\lambda_0^{-2}P(v_{\lambda_0})+\left(\frac{\lambda\alpha}{2}-\frac{\lambda^2\lambda_0^{-1}\alpha}{2}\right)|v(0)|^2\\
&= \frac{\lambda\alpha(\lambda_0-\lambda)}{2\lambda_0}|v(0)|^2.
\end{align*}
Hence, $P(v_{\lambda})>0$ if $\lambda<\lambda_0$ and $P(v_{\lambda})<0$ if $\lambda>\lambda_0$. This proves (iv). Finally, we have
\[
\frac{\partial^2}{\partial^2\lambda}S(v_{\lambda})=P(v)-\frac{\alpha}{2}|v(0)|^2<0.
\]
This proves (v).
\end{proof}
In the case of functions such that $v(0)=0$, we have the following lemma.
\begin{lemma}
\label{property of function under scaling 2}
Let $v \in H^1(\R^{+})\setminus \left\{0\right\}$, $v(0)=0$ and $P(v) = 0$ then we have
\[
S(v_{\lambda})=S(v) \quad \text{ for all } \lambda >0.
\]
\end{lemma}
\begin{proof}
The proof is simple, using the fact that
\[
\frac{\partial}{\partial\lambda}S(v_{\lambda})=\frac{1}{\lambda}P(v_{\lambda})=\lambda P(v)=0.
\]
Hence, we obtain the desired result.
\end{proof}

Now, consider the minimization problems
\begin{align}
d_{\mathcal{M}} &:=\inf\left\{S(v):v\in \mathcal{M}\right\},\\
m &:=\inf\left\{S(v),v\in H^1(\R^{+})\setminus 0,S'(v)=0\right\},\label{eq minimizer m}
\end{align}
where 
\[
\mathcal{M}=\left\{v\in H^1(\R^{+})\setminus 0, P(v)=0, K(v)\leq 0\right\}.
\]
By classical arguments, we can prove the following property.
\begin{proposition}\label{property of m}
Let $m$ be defined as above. Then, we have
\[
m= \inf\left\{S(v): v\in H^1(\R^{+})\setminus 0,K(v)=0\right\}.
\]
\end{proposition}
We have the following relation between the minimization problems $m$ and $d_{\mathcal{M}}$.
\begin{lemma}\label{lemma relation m and dM}
Let $m$ and $d_{\mathcal{M}}$ be defined as above. We have
\[
m=d_{\mathcal{M}}.
\]
\end{lemma}
\begin{proof}
Let $\mathcal{G}$ be the set of all minimizers of \eqref{eq minimizer m}. If $\varphi \in \mathcal{G}$ then $S'(\varphi)=0$. By the definition of $S$, $P$, $K$ we have $P(\varphi)=0$ and $K(\varphi)=0$. Hence, $\varphi\in \mathcal{M}$, this implies $S(\varphi) \geq d_{\mathcal{M}}$. Thus, $m \geq d_{\mathcal{M}}$. 

Conversely, let $v \in \mathcal{M}$. If $K(v)=0$ then $S(v) \geq m$, using Proposition \ref{property of m}. Otherwise, $K(v)<0$. Using the scaling $v_{\lambda}(x)=\lambda^{\frac{1}{2}}v(\lambda x)$, we have
\[
K(v_{\lambda})=\lambda^2\norm{v_x}^2_{L^2(\R^{+})}-\frac{3\lambda^2}{16}\norm{v}^6_{L^6(\R^{+})}+\omega\norm{v}^2_{L^2(\R^{+})}+\frac{\alpha\lambda}{2}|v(0)|^2 \rightarrow \omega\norm{v}^2_{L^2(\R^{+})}>0,
\]
as $\lambda \rightarrow 0$. Hence, $K(v_{\lambda})>0$ as $\lambda>0$ is small enough. Thus, there exists $\lambda_1 \in (0,1)$ such that $K(v_{\lambda_1})=0$. Using Proposition \ref{property of m}, $S(v_{\lambda_1}) \geq m$. We consider two cases. First, if $v(0)=0$ then using Lemma \ref{property of function under scaling 2}, we have $S(v)=S(v_{\lambda_1})\geq m$. Second, if $v(0)\neq 0$ then using Lemma \ref{property of function under scaling 1}, we have $S(v)\geq S(v_{\lambda_1})\geq m$. In any case, $S(v) \geq m$. This implies $d_{\mathcal{M}} \geq m$, and completes the proof.
\end{proof}

Define
\[
\mathcal{V}:=\left\{v \in H^1(\R^{+})\setminus \left\{0\right\}:K(v)<0,P(v)<0,S(v)<m\right\}.
\]
We have the following important lemma.
\begin{lemma}\label{invariant under flow of V}
If $v_0 \in \mathcal{V}$ then the solution $v$ of \eqref{eq first} associated with $v_0$ satisfies $v(t) \in \mathcal{V}$ for all $t$ in the time of existence. 
\end{lemma}
\begin{proof}
Since $S(v_0)<0$, by conservation of the energy and the mass we have
\begin{equation}\label{eq function S}
S(v)=E(v)+\omega M(v)=E(v_0)+\omega M(v_0)=S(v_0)<m.
\end{equation}
If there exists $t_0>0$ such that $K(v(t_0)) \geq 0$ then by continuity of $K$ and $v$, there exists $t_1 \in (0,t_0]$ such that $K(v(t_1))=0$. This implies $S(v(t_1)) \geq m$, using Proposition \ref{property of m}. This contradicts \eqref{eq function S}. Hence, $K(v(t))<0$ for all $t$ in the time of existence of $v$. Now, we prove $P(v(t))<0$ for all $t$ in the time of existence of $v$. Assume that there exists $t_2>0$ such that $P(v(t_2))\geq 0$, then, there exists $t_3 \in (0,t_2]$ such that $P(v(t_3))=0$. Using the previous lemma, $S(v(t_3)) \geq m$, which contradicts \eqref{eq function S}. This completes the proof.  
\end{proof}

Using the above lemma, we have the following property of solutions of \eqref{eq first} when the initial data lies on $\mathcal{V}$.

\begin{lemma}\label{lm3.16}
Let $v_0 \in \mathcal{V}$, $v$ be the corresponding solution of \eqref{eq first} in $(T_{min},T_{max})$. There exists $\delta>0$ independent of $t$  such that $P(v(t)) <-\delta$, for all $t \in (T_{min},T_{max})$. 
\end{lemma}  
\begin{proof}
Let $t \in (T_{min},T_{max})$, $u=v(t)$ and $u_{\lambda}(x)=\lambda^{\frac{1}{2}}u(\lambda x)$. Using Lemma \ref{property of function under scaling 1}, there exists $\lambda_0 \in (0,1)$ such that $P(u_{\lambda_0})=0$. If $K(u_{\lambda_0})\leq 0$ then we keep $\lambda_0$. Otherwise,  $K(u_{\lambda_0})>0$, then, there exists $\widetilde{\lambda}_0 \in (\lambda_0,1)$ such that $K(u_{\widetilde{\lambda}_0})=0$. We replace $\lambda_0$ by $\widetilde{\lambda}_0$. In any case, we have
\begin{equation}\label{eq bounded below 1}
S(u_{\lambda_0})\geq m.
\end{equation}
By (v) of Proposition \ref{property of function under scaling 1} we have
\[
S(u)-S(u_{\lambda_0}) \geq (1-\lambda_0) \frac{\partial}{\partial\lambda}S(u_{\lambda})\vert_{\lambda=1}=(1-\lambda_0)P(u).
\]
In addition $P(u)<0$, we obtain
\begin{equation}\label{eq bounded below of S}
S(u)-S(u_{\lambda_0}) \geq (1-\lambda_0)P(u)>P(u).
\end{equation}
Combined \eqref{eq bounded below 1} and \eqref{eq bounded below of S}, we obtain
\[
S(v_0)-m=S(v(t))-m=S(u)-m\geq S(u)-S(u_{\lambda_0}) > P(u)=P(v(t)).
\] 
Setting 
\[
-\delta:=S(v_0)-m,
\]
we obtain the desired result.
\end{proof}
Using the previous lemma, if the initial data lies on $\mathcal{V}$ and satisfies a weight condition then the associated solution blows up in finite time on $H^1(\R^{+})$. More precisely, we have the following result.  

\begin{proposition}\label{3.17}
Let $\varphi \in \mathcal{V}$ such that $|x|\varphi \in L^2(\R^{+})$. Then the corresponding solution $v$ of \eqref{eq first} blows up in finite time on $H^1(\R^{+})$.
\end{proposition}
\begin{proof}
By Lemma \ref{lm3.16}, there exists $\delta>0$ such that $P(v(t))<-\delta$ for $t \in (T_{min},T_{max})$. Remember that 
\begin{align}\label{eq derivative in time}
\frac{\partial}{\partial t}\norm{xv(t)}^2_{L^2(\R^{+})}&=J(t)-\int_{\R^{+}}x|v|^4\, dx,
\end{align}
where $J(t)$ satisfies
\[
\partial_tJ(t)=4\left(2\norm{v_x}^2_{L^2(\R^{+})}-\frac{1}{8}\norm{v}^6_{L^6(\R^{+})}+\alpha|v(0)|^2\right)=8(P(v(t))) < -8\delta.
\]
This implies that 
\[
J(t)=J(0)+8\int_0^tP(v(s))\, ds < J(0)-8\delta t.
\]
Hence, from \eqref{eq derivative in time}, we have
\begin{align*}
\norm{xv(t)}^2_{L^2(\R^{+})} &= \norm{xv(0)}^2_{L^2(\R^{+})}+\int_0^t J(s)\,ds-\int_0^t\int_{\R^{+}} x|v|^4\,dx\,ds\\
&\leq \norm{xv(0)}^2_{L^2(\R^{+})}+\int_0^t(J(0)-8\delta s)\, ds\\
&\leq \norm{xv(0)}^2_{L^2(\R^{+})}+J(0)t-4\delta t^2.
\end{align*}
Thus, for $t$ sufficiently large, there is a contradiction with $\norm{xv}_{L^2(\R^{+})}\geq 0$. Hence, $T_{max}<\infty$ and $T_{min}>-\infty$. By the blow up alternative, we have 
\[
\mathop{\lim}\limits_{t \rightarrow T_{max}}\norm{v_x}_{L^2(\R^{+})}=\mathop{\lim}\limits_{t \rightarrow T_{min}}\norm{v_x}_{L^2(\R^{+})}=\infty.
\]
This completes the proof.
\end{proof}
\begin{proof}[Proof of Theorem \ref{stable and instale of SDW for eq 2} (2).]
Using Proposition \ref{3.17}, we need to construct a sequence $(\varphi_n) \subset \mathcal{V}$ such that $\varphi_n$ converges to $\varphi$ in $H^1(\R^{+})$. Define 
\[
\varphi_{\lambda}(x)=\lambda^{\frac{1}{2}}\varphi(\lambda x).
\] 
We have 
\[
S(\varphi)=m, \quad P(\varphi)=K(\varphi)=0, \quad \varphi(0) \neq 0. 
\]
By (iv) of Proposition \ref{property of function under scaling 1},
\[
S(\varphi_{\lambda})<m \text{ for all } \lambda>0.
\]
In the addition, 
\[
P(\varphi_{\lambda})<0 \text{ for all } \lambda>1.
\]
Moreover, 
\begin{align*}
\frac{\partial}{\partial\lambda}K(\varphi_{\lambda})&=2\lambda\left(\norm{\varphi_x}^2_{L^2(\R^{+})}-\frac{3}{16}\norm{\varphi}^6_{L^6(\R^{+})}\right)+\alpha|\varphi(0)|^2\\
&=2\lambda(K(\varphi)-\omega\norm{\varphi}^2_{L^2(\R^{+})}-\alpha|\varphi(0)|^2)+\alpha|\varphi(0)|^2\\
&=-2\omega\lambda\norm{\varphi}^2_{L^2(\R^{+})}-\alpha(2\lambda-1)|\varphi(0)|^2\\
&<0,
\end{align*}
when $\lambda>1$. Thus, $K(\varphi_{\lambda})<K(\varphi)=0$ when $\lambda>1$. This implies $\varphi_{\lambda} \in \mathcal{V}$ when $\lambda>1$. Let $\lambda_n>1$ such that $\lambda_n \rightarrow 1$ as $n \rightarrow \infty$. Define, for $n \in \N^{*}$
\[
\varphi_n=\varphi_{\lambda_n},
\]
then, the sequence $(\varphi_n)$ satisfies the desired property. This completes the proof of Theorem \ref{stable and instale of SDW for eq 2}.
\end{proof}
 \bibliographystyle{abbrv}
 \bibliography{repairlibraryforpertubationschrodingerderivative}

\end{document}